\documentclass[12pt]{amsart}

\usepackage{amssymb, amsmath, amsthm}
\usepackage[foot]{amsaddr}
\usepackage[text={6in,9in},centering]{geometry}
\usepackage[misc]{ifsym}

\newcommand{\pD}[2]{\frac{\partial #1}{\partial #2}}
\newcommand{\rD}[2]{\frac{d #1}{d #2}}

\newcommand{\vn}[1]{\lVert#1\rVert}
\newcommand{\IP}[2]{\left< #1 , #2 \right>}

\newcommand{\R}{\ensuremath{\mathbb{R}}}

\newcommand{\SW}{\ensuremath{\mathcal{W{}}}}

\newcommand{\SH}{\ensuremath{\mathcal{H{}}}}

\newcommand{\BW}{\ensuremath{\text{\bf W}}}
\newcommand{\BH}{\ensuremath{\text{\bf H}}}

\newtheorem{thm}{Theorem}{\bf}{\it}
\newtheorem{cor}[thm]{Corollary}{\bf}{\it}
\newtheorem{prop}[thm]{Proposition}{\bf}{\it}
\newtheorem{lem}[thm]{Lemma}{\bf}{\it}
{\bf}{\it}
{\bf}{\it}
{\bf}{\it}
\newtheorem*{defn}{Definition}{\bf}{\rm}
\newtheorem*{rmk}{Remark}{\bf}{\rm}

\title{A classification theorem for Helfrich surfaces}

\author{James McCoy$^1$}
\author{Glen Wheeler$^2$}
\address{$^1$James McCoy\\
         Institute for Mathematics and Applied Statistics\\
         University of Wollongong\\
         Northfields Ave\\
         Wollongong, NSW 2500, Australia
}
\address{$^2$Glen Wheeler (\Letter)\\
         Institut f\"ur Analysis und Numerik\\
         Otto-von-Guericke-Universit\"at\\
         Postfach 4120\\
	 D-39016 Magdeburg, Germany
}
\thanks{Financial support for Glen Wheeler from the Alexander-von-Humboldt Stiftung is gratefully acknowledged}
\thanks{\emph{Email address:} {\tt wheeler@ovgu.de}}

\begin{document}

\begin{abstract}
In this paper we study the functional $\SW_{\lambda_1,\lambda_2}$, which is the
the sum of the Willmore energy, $\lambda_1$-weighted surface area, and $\lambda_2$-weighted volume, for surfaces
immersed in $\R^3$.
This coincides with the Helfrich functional with zero `spontaneous curvature'.
Our main result is a complete classification of all smooth immersed critical points of the functional with
$\lambda_1\ge0$ and small $L^2$ norm of tracefree curvature.
In particular we prove the non-existence of critical points of the functional for which the surface area and enclosed
volume are positively weighted.
\end{abstract}

\keywords{global differential geometry \and fourth order \and geometric analysis \and higher order elliptic partial
differential equations}
\subjclass[2000]{35J30 \and 58J05 \and 35J62 }

\maketitle


\section{Introduction}

Consider a surface $\Sigma$ immersed via a smooth immersion $f:\Sigma\rightarrow\R^3$.
There are several functionals of $f$ relevant to our study:
\begin{align*}
\mu(\Sigma) &= \int_\Sigma d\mu
\text{, the surface area}\\
\text{Vol $\Sigma$} &= \int_{\Sigma} f^*(d\SH^3)
\text{, the signed enclosed volume}\\
\SW(f) &= \frac{1}{4}\int_\Sigma |H|^2d\mu
\text{, the Willmore energy}\\
\SH^{c_0}_{\lambda_1,\lambda_2}(f)
       &= \frac{1}{4}\int_\Sigma (H-c_0)^2d\mu + \lambda_1\mu(\Sigma) + \lambda_2\text{Vol $\Sigma$}
\text{, the Helfrich energy.}
\end{align*}
In the above we have used $d\mu$ to denote the area element induced by $f$ on $\Sigma$, $d\SH^3$ to denote Hausdorff measure
in $\R^3$, $H$ to denote the mean curvature, and $c_0,\lambda_1,\lambda_2$ are real numbers.
When $f:\Sigma\rightarrow\R^3$ is not an embedding, we still require the definition of the enclosed volume to make
sense.
For this reason we have used the pull-back of the Euclidean volume element $f^*(d\SH^3)$ by the map $f$, as is standard
(see \cite[Lemma 2.1]{BdCE88} for example).
Our notation is further clarified in Section 2.

The Helfrich functional $\SH^{c_0}_{\lambda_1,\lambda_2}$ is of great interest in applications.
Although it appears to have first been studied by Schadow \cite{S22}, its modern form and subsequent popularity is due
to Helfrich \cite{H73} (just as the Willmore functional was considered by Germain \cite{B80sophie,G21}, Poisson
\cite{P1812}, Thomsen \cite{T23}, and others, long before the time of Willmore \cite{W65}), who famously proposed that
the minimisers of the functional model the shape of an elastic lipid bilayer, such as a biomembrane.
Since this time the model has enjoyed considerable popularity.  Despite this, relatively few analytical results can be
found in the literature.

Critical points of the functional $\SH^{c_0}_{\lambda_1,\lambda_2}$ are called \emph{Helfrich surfaces} and solve the
Euler-Lagrange equation
\[
\BH^{c_0}_{\lambda_1,\lambda_2}(f) := -\bigg(\Delta H + H|A^o|^2 + c_0\Big(2K-\frac{1}{2}Hc_0\Big)
                  - 2\lambda_1H - 2\lambda_2\bigg)
 = 0,
\]
where $A^o$ and $K$ denote the tracefree second fundamental form and Gauss curvature of $f$ respectively.
For the reader's convenience we briefly derive this equation in Lemma \ref{LMeveq}.

Consider the functional $\SW_{\lambda_1,\lambda_2} = \SH^0_{\lambda_1,\lambda_2}$.
The functional $\SW_{\lambda_1,\lambda_2}$ differs from the Willmore functional $\SW = \SW_{0,0}$ by the surface area
and volume terms only, which serve to punish or reward additional surface area and volume depending on the sign of
$\lambda_1$ and $\lambda_2$ respectively.
Critical points of $\SW_{\lambda_1,\lambda_2}$ satisfy
\begin{equation}
\BW_{\lambda_1,\lambda_2}(f)
 := \Delta H + H|A^o|^2 - 2\lambda_1H - 2\lambda_2
 = 0.
\label{EQglW}
\end{equation}
Clearly the effects of $\lambda_1$ and $\lambda_2$ on the solution $f$ are local in nature.
For this reason we name $\SW_{\lambda_1,\lambda_2}$ the \emph{locally constrained Willmore functional}.
From a physical perspective, our simplification corresponds with an assumption that the fluid surrounding the membrane
$f(\Sigma)$ and the fluid contained inside the membrane $f(\Sigma)$ induce zero spontaneous curvature in $f(\Sigma)$.
Thus solutions of \eqref{EQglW} faithfully represent lipid bilayers in certain settings.

Heuristically, one expects the sign of $\lambda_1, \lambda_2$ to have a dramatic impact on the nature of the minimisers
of $\SW_{\lambda_1,\lambda_2}$.
If $\lambda_1$ is negative, the functional is unbounded from below.
Furthermore, as we assume $f$ is only immersed, one can not control the sign of the volume term.
This makes a direct minimisation procedure quite difficult to carry out.
Despite this, the presence of the Willmore term punishes self-intersections and other complicated branching behaviour.
One may thus hope to understand the structure of the family of immersed critical points with Willmore energy slightly
larger than that of an embedded round sphere.

Our main result is a direct classification of all smooth properly immersed solutions to the Euler-Lagrange equation
\eqref{EQglW} with Willmore energy close to that of the sphere.
This is more general than classifying the set of all immersed critical points of $\SW_{\lambda_1,\lambda_2}$ under the
energy condition, as we do not assume $f$ is \emph{closed}.
In this case the functional $\SW_{\lambda_1,\lambda_2}$ is not well-defined.
Indeed any reasonable definition of the functional for proper non-closed immersed surfaces $f$ would set
$\SW_{\lambda_1,\lambda_2}(f) = \infty$.
We may nevertheless work directly with the differential equation $\BW_{\lambda_1,\lambda_2}(f) = 0$.
In this manner one may speak of Helfrich surfaces with unbounded area, not well-defined volume, and so on.
The main motivation for considering this class of surfaces is that one expects proper non-closed immersed surfaces to
arise in the analysis of singularities, through a blowup procedure for example.
\begin{thm}
\label{Tgap}
Suppose $f:\Sigma\rightarrow\R^3$ is a smooth properly immersed surface.
If
\begin{equation}
\int_\Sigma |A^o|^2 d\mu < \frac1{2c_1c_2} := \varepsilon_1
\label{Egapass}
\end{equation}
where $c_1$, $c_2$ are the absolute constants defined in Corollary \ref{CYgapcor} and Proposition \ref{Pgapest}
respectively,
the following statements hold: ($\lambda_1 > 0$)
\begin{align*}
(\lambda_2 < 0)
&\hskip+6mm
 \BW_{\lambda_1,\lambda_2}(f) = 0
\text{ if and only if }
 f(\Sigma) = S_{-\frac{2\lambda_1}{\lambda_2}}(x)\text{ for some $x\in\R^3$},
\\
(\lambda_2 = 0)
&\hskip+6mm
 \BW_{\lambda_1,\lambda_2}(f) = 0
\text{ if and only if }
 f(\Sigma)\text{ is a plane},
\\
(\lambda_2 > 0)
&\hskip+6mm
 \BW_{\lambda_1,\lambda_2}(f) \ne 0.
\intertext{If $\lambda_1 = 0$ then}
(\lambda_2 = 0)
&\hskip+6mm
 \BW_{\lambda_1,\lambda_2}(f) = 0
\text{ if and only if }
 f(\Sigma)\text{ is a plane or a sphere},
\\
(\lambda_2 \ne 0)
&\hskip+6mm
 \BW_{\lambda_1,\lambda_2}(f) \ne 0.
\end{align*}
Here $S_\rho(x) = \partial B_\rho(x)$ denotes the sphere of radius $\rho$ centred at $x\in\R^3$.
\end{thm}
We note that the fourth statement above is an improvement of \cite[Theorem 2.7]{KS01} for the case where $n=3$.


\begin{rmk}
The catenoid $f_c:\Sigma\rightarrow\R^3$ is a properly immersed minimal surface with
\[
\int_\Sigma |A^o|^2 d\mu = 8\pi.
\]
Furthermore $\BW_{\lambda_1,0}(f_c) = 0$ for all $\lambda_1\in\R$ and so Theorem \ref{Tgap} holds at best for
$\varepsilon_1\le8\pi$.
\end{rmk}

The methods we use in this paper are inspired by some recent progress on the analysis of the Willmore functional
\cite{KS01,KS02,KS04} due to Kuwert and Sch\"atzle.
There are notable differences between the functional $\SW_{\lambda_1,\lambda_2}$ and the Willmore functional
$\SW$.
The set of all critical points contain the set of all minimisers of the functional, and the functional
$\SW_{\lambda_1,\lambda_2}$ is not bounded from below.
Despite $\lambda_1\ge0$, one is not able to control the sign of the volume term.
It is conceivable that an exotic critical point with enormous negative volume exists as the limit of a properly immersed
minimising sequence.

We prove here that the condition \eqref{Egapass} prevents any critical points from appearing which are not flat planes
or round spheres.
The operator $\BW_{\lambda_1,\lambda_2}$ does not admit a maximum principle, and thus we do not have access to the large
assortment of tools it brings.
We instead rely throughout the paper on estimates for curvature quantities on smooth immersed surfaces combined with the
divergence theorem and the Michael-Simon Sobolev inequality \cite{MS73}.
Using these we derive a series of local integral estimates, which when globalised, yield Theorem \ref{Tgap}.
This is similar to the idea used to prove \cite[Theorem 2.7]{KS01}.

It is notable that the theorem does not require any smallness of the parameters $\lambda_1$ and $\lambda_2$, and does
not require any smallness (or even boundedness) of $\vn{H}_p$ for any $p$.
Indeed, the surface $f$ may a priori possess quite wild behaviour at infinity.
This implies in particular that these critical points of the functional may have arbitrarily large or not even
well-defined Helfrich energy.
Existence results without any smallness condition at all are extremely difficult and only known in the class of
rotationally symmetric critical points for the Willmore functional \cite{DFGS11,DG09}.
There one finds such a plethora of solutions that the classification question is for the moment too difficult.

This paper is organised as follows.
In Section 2 we set up our notation and briefly derive the first variation of the functional
$\SH^{c_0}_{\lambda_1,\lambda_2}$.
In Section 3 we establish local integral estimates for the Euler-Lagrange operator $\BW_{\lambda_1,\lambda_2}$ in $L^2$.
In Section 4 we prove Theorem \ref{Tgap}.
Finally, we have included several proofs and derivations of known results in Appendix A for the convenience of the
reader.

The authors would each like to thank their home institutions for their support and their collaborator's home
institutions for their hospitality during respective visits.  Both authors would also like to thank Prof.
Graham Williams for useful discussions during the preparation of this work.


\section*{Acknowledgements}

The research of the first author was supported under the Australian Research Council's Discovery Projects scheme
(project number DP120100097).
The first author is also grateful for the support of the University of Wollongong Faculty of Informatics Research
Development Scheme grant.

Part of this work was carried out while the second author was a research associate supported by the Institute for
Mathematics and Its Applications at the University of Wollongong.
Part of this work was also carried out while the second author was a Humboldt research fellow at the Otto-von-Guericke
Universit\"at Magdeburg.  The support of the Alexander von Humboldt Stiftung is gratefully acknowledged.


\section{Preliminaries}

We consider a surface $\Sigma$ immersed in $\R^3$ via $f:\Sigma\rightarrow\R^3$ and endow a Riemanain metric on $\Sigma$
defined componentwise by
\begin{equation}
g_{ij} = \IP{\partial_if}{\partial_jf},
\label{EQmetric}
\end{equation}
where $\partial$ denotes the regular partial derivative and $\IP{\cdot}{\cdot}$ is the standard Euclidean inner product.
That is, we consider the Riemannian structure on $\Sigma$ induced by $f$, where in particular the metric $g$ is given by
the pull-back of the standard Euclidean metric along $f$.  Integration on $\Sigma$ is performed with respect to the
induced area element
\begin{equation}
d\mu = \sqrt{\text{det }g}\ d\SH^3,
\label{EQdmu}
\end{equation}
where $d\SH^3$ is the standard Hausdorff measure on $\R^3$.

The metric induces an inner product structure on all tensor fields defined over $\Sigma$, where corresponding pairs of
indices are contracted.  For example, if $T$ and $S$ are $(1,2)$ tensor fields,
\[
\IP{T}{S}_g = g_{ip}g^{jq}g^{kr}T^{i}_{jk}S^p_{qr},\qquad |T|^2 = \IP{T}{T}_g.
\]
In the above, and in what follows, we shall use the summation convention on repeated indices unless otherwise explicitly
stated.

The \emph{second fundamental form} $A$ is a symmetric $(0,2)$ tensor field over $\Sigma$ with components
\begin{equation}
A_{ij} = \IP{\partial^2_{ij}f}{\nu},
\label{EQsff}
\end{equation}
where $\nu$ is an inward pointing unit vector field normal along $f$.
With this choice one finds that the second fundamental form of the standard round sphere embedded in $\R^3$ is positive.
There are two invariants of $A$ relevant to our work here: the first is the trace with respect to the metric
\[
H = \text{trace}_g\ A = g^{ij}A_{ij}
\]
called the \emph{mean curvature}, and the second the determinant with respect to the metric, called the \emph{Gauss
curvature},
\[
K = \text{det}_g\ A
 = \text{det }\big(g^{ik}A_{kj}\big),
\]
where $\big(P^{ik}Q_{kj}\big)$ is used above to denote the matrix with $i,j$-th component equal to $P^{ik}Q_{kj}$.

The mean and Gauss curvatures are easily expressed in terms of the principal curvatures: at a
single point we may make a local choice of frame for the tangent bundle $T\Sigma$ under which the eigenvalues of $A$
appear along its diagonal.  These are denoted by $k_1$, $k_2$ and are called the \emph{principal curvatures}.  We then
have
\[
H = k_1+k_2,\qquad K = k_1k_2.
\]
We shall often decompose the second fundamental form into its trace and its tracefree parts,
\[
A = A^o + \frac{1}{2}gH,
\]
where $(0,2)$ tensor field $A^o$ is called the \emph{tracefree second fundamental form}.  In a basis which diagonalises
$A$, a so-called principal curvature basis, its norm is given by
\[
|A^o|^2 = \frac{1}{2}(k_1-k_2)^2.
\]
The Christoffel symbols of the induced connection are determined by the metric,
\begin{equation*}
\Gamma_{ij}^k = \frac{1}{2}g^{kl}
                \left(\partial_ig_{jl} + \partial_jg_{il} - \partial_lg_{ij}\right),
\end{equation*}
so that then the covariant derivative on $\Sigma$ of a vector $X$ and of a covector $Y$ is
\begin{align*}
\nabla_jX^i &= \partial_jX^i + \Gamma^i_{jk}X^k\text{, and}\\
\nabla_jY_i &= \partial_jY_i - \Gamma^k_{ij}Y_k
\end{align*}
respectively.

From \eqref{EQsff} and the smoothness of $f$ we see that the second fundamental form is
symmetric; less obvious but equally important is the symmetry of the first covariant derivatives of
$A$, 
\[ \nabla_iA_{jk} = \nabla_jA_{ik} = \nabla_kA_{ij}, \]
commonly referred to as the Codazzi equations.

One basic consequence of the Codazzi equations which we shall make use of is that the gradient of the mean curvature is
completely controlled by a contraction of the $(0,3)$ tensor $\nabla A^o$.  To see this, first note that
\[
\nabla_i A^i_j = \nabla_i H = \nabla_i \Big( (A^o)^i_j + \frac12 g_j^i H\Big),
\]
then factorise to find
\begin{equation}
\label{EQbasicgradHgradAo}
\nabla_j H = 2\nabla_i (A^o)^i_j =: 2(\nabla^* A^o)_j.
\end{equation}
This in fact shows that all derivatives of $A$ are controlled by derivatives of $A^o$.
For a $(p,q)$ tensor field $T$, let us denote by $\nabla_{(n)}T$ the tensor field with components
 $\nabla_{i_1\ldots i_n}T_{j_1\ldots j_q}^{k_1\ldots k_p} =
  \nabla_{i_1}\cdots\nabla_{i_n}T_{j_1\ldots j_q}^{k_1\ldots k_p}$.
In our notation, the $i_n$-th covariant derivative is applied first.
Since
\[
\nabla_{(k)} A = \Big(\nabla_{(k)} A^o + \frac12 g \nabla_{(k)} H\Big)
               = \Big(\nabla_{(k)} A^o +         g \nabla_{(k-1)}\nabla^*A^o\Big),
\]
we have
\begin{equation}
\label{EQbasicgradHgradAo2}
|\nabla_{(k)}A|^2 \le 3|\nabla_{(k)}A^o|^2.
\end{equation}
The fundamental relations between components of the Riemann curvature tensor $R_{ijkl}$, the Ricci
tensor $R_{ij}$ and scalar curvature $R$ are given by Gauss' equation
\begin{align*}
R_{ijkl} &= A_{ik}A_{jl} - A_{il}A_{jk},\intertext{with contractions}
g^{jl}R_{ijkl}
   = R_{ik} 
  &= HA_{ik} - A_i^jA_j^k\text{, and}\\
g^{ik}R_{ik}
   = R
  &= |H|^2 - |A|^2.
\end{align*}
We will need to interchange covariant derivatives; for vectors $X$ and covectors $Y$ we obtain
\begin{align}
\notag
\nabla_{ij}X^h - \nabla_{ji}X^h &= R^h_{ijk}X^k = (A_{lj}A_{ik}-A_{lk}A_{ij})g^{hl}X^k,\\
\label{EQinterchangespecific}
\nabla_{ij}Y_k - \nabla_{ji}Y_k &= R_{ijkl}g^{lm}Y_m = (A_{lj}A_{ik}-A_{il}A_{jk})g^{lm}Y_m.
\end{align}
We also use for tensor fields $T$ and $S$ the notation $T*S$ (as in Hamilton \cite{H82}) to denote a linear combination
of new tensors, each formed by contracting pairs of indices from $T$ and $S$ by the metric $g$ with multiplication by a
universal constant.  The resultant tensor will have the same type as the other quantities in the expression it appears.
As is common for the $*$-notation, we slightly abuse this constant when certain subterms do not appear in our $P$-style
terms. For example
\begin{align*}
|\nabla A|^2
    = \IP{\nabla A}{\nabla A}_g
    = 1\cdot\left(\nabla_{(1)}A*\nabla_{(1)}A\right) + 0\cdot\left(A*\nabla_{(2)}A\right)
    = P_2^2(A).
\end{align*}

The Laplacian we will use is the Laplace-Beltrami operator on $\Sigma$, with the components of $\Delta T$
given by
\[
\Delta T_{j_1\ldots j_q}^{k_1\ldots k_p} = g^{pq}\nabla_{pq}T_{j_1\ldots j_q}^{k_1\ldots k_p}  = \nabla^p\nabla_pT_{j_1\ldots j_q}^{k_1\ldots k_p}.
\]
Using the Codazzi equation with the interchange of covariant derivative formula given above, we
obtain Simons' identity:
\begin{align}
g^{kl}\nabla_{ki} A_{lj} &= g^{kl}\nabla_{ik}A_{lj} + g^{kl}g^{pq}R_{kilp}A_{qj} + g^{kl}g^{pq}R_{kijp}A_{lq}
 \notag\\
\Delta A_{ij} &= \nabla_{ik}A_{j}^k
 + g^{kl}g^{pq} (A_{pi}A_{kl}-A_{kp}A_{il}) A_{qj}
\notag\\&\quad
 + g^{kl}g^{pq} (A_{pi}A_{kj}-A_{kp}A_{ij}) A_{lq}
 \notag\\
              &= \nabla_{ik}A_{j}^k
 + Hg^{pq}A_{pi}A_{qj}
 - g^{kl}g^{pq} A_{kp}A_{il} A_{qj}
\notag\\&\quad
 + g^{kl}g^{pq} A_{pi}A_{kj} A_{lq}
 - A_{ij}\IP{A}{A}_g
 \notag\\
\label{EQsi}  &= \nabla_{ij}H + HA_{i}^kA_{kj} - |A|^2A_{ij},
\end{align}
or in $*$-notation
\[
\Delta A = \nabla_{(2)}H + A*A*A.             
\]
The interchange of covariant derivatives formula for mixed tensor fields $T$ is simple to state in $*$-notation:
\begin{equation}
\label{EQinterchangegeneral}
\nabla_{ij}T = \nabla_{ji}T + T*A*A.
\end{equation}

We now briefly compute the first variation of the Helfrich functional.

\begin{lem}
Suppose $f:\Sigma\rightarrow\R^3$ is a closed immersed surface and $\phi:\Sigma\rightarrow \R^3$ is a vector field
normal along $f$.  Then
\begin{align*}
\rD{}{t}&\SH^{c_0}_{\lambda_1,\lambda_2}(f+t\,\phi)\bigg|_{t=0}
\\
 &= \frac{1}{2}\int_\Sigma \IP{\phi}{\nu}(\Delta H + H|A^o|^2+2c_0K-(2\lambda_1+c_0^2/2)H-2\lambda_2) d\mu.
\end{align*}
In particular, if 
\[
 \BH^{c_0}_{\lambda_1,\lambda_2}(f)
 := -\big(\Delta H + H|A^o|^2+2c_0K-(2\lambda_1+c_0^2/2)H-2\lambda_2\big)
  = 0
\]
then
\[
\rD{}{t}\SH^{c_0}_{\lambda_1,\lambda_2}(f+t\phi)\bigg|_{t=0}
 = 0
\]
and $f$ is a critical point of $\SH^{c_0}_{\lambda_1,\lambda_2}$.
\label{LMeveq}
\end{lem}
\begin{proof}
For the sake of brevity we shall suppress the dependence of various quantities on $(f+t\,\phi)$.
All derivatives with respect to $t$ will be implicitly evaluated at $t=0$.
Let us assume $\phi = \varphi\,\nu$ for some function $\varphi:\Sigma\rightarrow\R$.
From the formulae \eqref{EQmetric}, \eqref{EQsff} we compute the variation of the mean curvature as
\[
\pD{}{t} H = g^{ij}\pD{}{t}A_{ij} + A_{ij}\pD{}{t}g^{ij}
             = \Delta H - \phi|A|^2 + 2\phi|A|^2
             = \Delta \varphi + \varphi |A|^2,
\]
while for the induced area element differentiating the determinant in \eqref{EQdmu} gives
\[
\rD{}{t} d\mu = -\varphi H\, d\mu.
\]
Each of these formulae are proven in \cite[Appendix A]{E04}.
Using these and the divergence theorem we compute
\[
\rD{}{t}\int_\Sigma d\mu
 = -\int_\Sigma \varphi H\, d\mu,
\]
\[
\rD{}{t}\int_\Sigma H\, d\mu
 = -2\int_\Sigma \varphi K\, d\mu,
\]
and
\[
\frac{1}{2}\rD{}{t}\int_\Sigma |H|^2 d\mu
 = \int_\Sigma \varphi (\Delta H + H|A^o|^2)\, d\mu.
\]
The variation of the enclosed volume may be computed as in \cite{BdCE88}.
For the convenience of the reader we outline the argument here.
Let $p\in\Sigma$.
Suppose $\{e_1,e_2,n\}$ is a positively oriented adapted orthonormal frame around $f(p)$.
Then
\[
f^*(d\SH^3)
 = a(t,p)\, dt \wedge d\mu
\]
where
\begin{align*}
a(t,p) 
 &= f^*(d\SH^3)\Big(\pD{}{t},e_1,e_2\Big)
  = d\SH^3\Big(\pD{f}{t},\pD{f}{e_1},\pD{f}{e_2}\Big)
\\
 &= \text{Vol}^{\R^3}\Big(\pD{f}{t},\pD{f}{e_1},\pD{f}{e_2}\Big)
  = \IP{\phi}{n}.
\end{align*}
We thus have
\begin{align*}
\rD{}{t}\text{Vol $\Sigma$}\bigg|_{t=0}
 &= \rD{}{t}\int_{\Sigma} a(t,p)\, dt\wedge d\mu\bigg|_{t=0}
\\
 &= \int_{\Sigma} a(0,p)\, d\mu
  = -\int_\Sigma \IP{\phi}{\nu}\, d\mu
  = -\int_\Sigma \varphi\, d\mu.
\end{align*}
Putting these together we have
\begin{align*}
\rD{}{t}
       \bigg(\frac{1}{4}&\int_\Sigma (H-c_0)^2d\mu + \lambda_1\mu(\Sigma) + \lambda_2\text{Vol
$\Sigma$}\bigg)\bigg|_{t=0}
\\
 &= \frac{1}{2}\int_\Sigma \varphi (\Delta H + H|A^o|^2)\, d\mu
    + \frac{c_0}{2}\int_\Sigma 2\varphi K\, d\mu
\\
 &\hskip+8mm
    - \big(\lambda_1+c_0^2/4\big)\int_\Sigma \varphi H\, d\mu
    - \lambda_2 \int_\Sigma \varphi\, d\mu
\\
 &= \frac{1}{2}\int_\Sigma \varphi\big(\Delta H + H|A^o|^2+2c_0K-(2\lambda_1+c_0^2/2)H-2\lambda_2\big)\, d\mu,
\end{align*}
which is the first statement of the lemma.
The remaining statement clearly follows from the first.
\end{proof}


\section{Control of geometry by the Euler-Lagrange operator}

In this section we prove estimates which give us control of the geometry of a surface in terms of
the Euler-Lagrange operator $\BW_{\lambda_1,\lambda_2}(f)$ and some error terms.
The geometric quantity which we aim to control is
\begin{equation*}
|\nabla_{(2)}A|^2 + |\nabla A|^2|A|^2 + |A|^4|A^o|^2.
\end{equation*}
Our control is gained through the use of the divergence theorem and multiplicative Sobolev
inequalities, as in \cite{KS01,W10sd}.
The key difference here is that the Euler-Lagrange operator is much more complicated and we must deal with several
extra terms.  Through the usage of a novel test function we have managed to exploit this extra complexity (in the form
of a lack of scale invariance of the functional).

\begin{defn}
Set $\gamma=\tilde{\gamma}\circ f:\Sigma\rightarrow[0,1]$,
$\tilde{\gamma} \in C^2_c(\R^3)$ satisfying
\begin{equation}
\vn{\nabla\gamma}_\infty \le c_{\gamma},\quad
\vn{\nabla_{(2)}\gamma}_\infty \le c_{\gamma}(c_\gamma+|A|),
\label{Egamma}
\tag{$\gamma$}
\end{equation}
for some absolute constant $c_{\gamma} < \infty$.
\end{defn}

\begin{rmk}
There is a smooth cutoff function $\tilde{\gamma}$  on $B_{2\rho}(y)\subset\R^3$, $y\in\R^3$, i.e.
\[
\tilde{\gamma}(x) = \begin{cases} 1 &\text{ if }|x-y|\le\rho \\
                                  0 &\text{ if }|x-y|\ge2\rho, 
\end{cases}
\quad\text{ and }\quad 0 \le \tilde{\gamma}(x) \le 1,
\]
such that $c_\gamma \le \frac{c}{\rho}$ for an absolute constant $c$ (see the appendix of \cite{W93}).
\end{rmk}

\begin{lem}
\label{LMgapest1}
Suppose $f:\Sigma\rightarrow\R^3$ is a smooth immersed surface and $\gamma$ is as in \eqref{Egamma}.  Then for $\delta>0$,
\begin{align*}
(1-2\delta)&\int_\Sigma |\Delta H|^2\gamma^4d\mu
   + 2\lambda_1\int_\Sigma |\nabla H|^2 \gamma^4d\mu
\\
&\le   \int_\Sigma \BW_{\lambda_1,\lambda_2}(f)\big(4\IP{\nabla H}{\nabla \gamma} + \gamma\Delta H\big) \gamma^3 d\mu
 +  \delta \int_\Sigma |H|^4|A^o|^2\gamma^4 d\mu
\\
&\quad
+ 8\frac{(c_\gamma)^2}{\delta}\int_\Sigma |\nabla H|^2 \gamma^2 d\mu
 + \frac{(4\delta^2+1)^2}{(4\delta)^3} \int_\Sigma |A^o|^6 \gamma^4d\mu.
\end{align*}
\end{lem}
\begin{proof}
Since
\begin{multline*}
- 2\lambda_1H\IP{\nabla H}{\nabla \gamma}\gamma^3 - 2\lambda_2\IP{\nabla H}{\nabla \gamma}\gamma^3
\\
=   \BW_{\lambda_1,\lambda_2}(f)\IP{\nabla H}{\nabla \gamma}\gamma^3 - \Delta H\IP{\nabla H}{\nabla \gamma}\gamma^3 - H|A^o|^2\IP{\nabla H}{\nabla \gamma}\gamma^3
\end{multline*}
we have for $\delta>0$
\begin{align}
- 2\lambda_1&\int_\Sigma H\IP{\nabla H}{\nabla \gamma}\gamma^3d\mu - 2\lambda_2\int_\Sigma\IP{\nabla H}{\nabla
  \gamma}\gamma^3d\mu
\notag
\\
\label{EQgle1}
&\le   \int_\Sigma \BW_{\lambda_1,\lambda_2}(f)\IP{\nabla H}{\nabla \gamma}\gamma^3d\mu
 + \delta \int_\Sigma |\Delta H|^2\gamma^4d\mu
\\
\notag
&\quad + \delta \int_\Sigma |H|^2|A^o|^4\gamma^4 d\mu
 + \frac{(c_\gamma)^2}{2\delta}\int_\Sigma |\nabla H|^2 \gamma^2 d\mu.
\end{align}
Multiplying \eqref{EQglW} by $\Delta H\,\gamma^4$ gives
\begin{equation*}
|\Delta H|^2\gamma^4 
 = \BW_{\lambda_1,\lambda_2}(f)\Delta H \gamma^4 - H|A^o|^2\Delta H \gamma^4 + 2\lambda_1H \Delta H \gamma^4
 + 2\lambda_2\Delta H \gamma^4,
\end{equation*}
which combined with the divergence theorem yields
\begin{align}
\notag
\int_\Sigma |\Delta H|^2\gamma^4d\mu
&=   \int_\Sigma \BW_{\lambda_1,\lambda_2}(f)\Delta H \gamma^4d\mu
   - \int_\Sigma H|A^o|^2\Delta H \gamma^4d\mu
\\
\notag
&\quad
   + 2\lambda_1\int_\Sigma H \Delta H \gamma^4d\mu
   + 2\lambda_2\int_\Sigma \Delta H \gamma^4d\mu
\\
\notag
&=   \int_\Sigma \BW_{\lambda_1,\lambda_2}(f)\Delta H \gamma^4d\mu
   - \int_\Sigma H|A^o|^2\Delta H \gamma^4d\mu
\\
\notag
&\quad
   - 2\lambda_1\int_\Sigma |\nabla H|^2 \gamma^4d\mu
   - 8\lambda_1\int_\Sigma H\IP{\nabla H}{\nabla \gamma} \gamma^3 d\mu
\\
\notag
&\quad
   - 8\lambda_2\int_\Sigma \IP{\nabla H}{\nabla \gamma} \gamma^3 d\mu.
\end{align}
Using \eqref{EQgle1} to estimate the inner product terms,
\begin{align*}
(1-2\delta)&\int_\Sigma |\Delta H|^2\gamma^4d\mu
   + 2\lambda_1\int_\Sigma |\nabla H|^2 \gamma^4d\mu
\\
&\le   \int_\Sigma \BW_{\lambda_1,\lambda_2}(f)\big(4\IP{\nabla H}{\nabla \gamma} + \gamma\Delta H\big) \gamma^3 d\mu
 +  \delta \int_\Sigma |H|^2|A^o|^4\gamma^4 d\mu
\\
&\quad
 + 8\frac{(c_\gamma)^2}{\delta}\int_\Sigma |\nabla H|^2 \gamma^2 d\mu
 +  \frac{1}{4\delta} \int_\Sigma |H|^2|A^o|^4 \gamma^4d\mu
\\
&\le   \int_\Sigma \BW_{\lambda_1,\lambda_2}(f)\big(4\IP{\nabla H}{\nabla \gamma} + \gamma\Delta H\big) \gamma^3 d\mu
 +  \delta \int_\Sigma |H|^4|A^o|^2\gamma^4 d\mu
\\
&\quad
 + 8\frac{(c_\gamma)^2}{\delta}\int_\Sigma |\nabla H|^2 \gamma^2 d\mu
 + \frac{(4\delta^2+1)^2}{(4\delta)^3} \int_\Sigma |A^o|^6 \gamma^4d\mu.
\end{align*}
This proves the lemma.
\end{proof}

\begin{lem}
Suppose $f:\Sigma\rightarrow\R^3$ is a smooth immersed surface and $\gamma$ is as in \eqref{Egamma}.  Then
\begin{align*}
&\int_\Sigma \big(|\nabla_{(2)}H|^2
                + |H|^2|\nabla H|^2
                + |H|^4|A^o|^2\big)\gamma^4d\mu
   + 2\lambda_1\int_\Sigma |\nabla H|^2 \gamma^4d\mu
\\
&\le 
    c\int_\Sigma \BW_{\lambda_1,\lambda_2}(f)\big(4\IP{\nabla H}{\nabla \gamma} + \gamma\Delta H\big) \gamma^3 d\mu
  + c(c_\gamma)^2\int_\Sigma |\nabla H|^2 \gamma^2d\mu
\\
&\quad
  + c\int_\Sigma\big(|A^o|^6 + |A^o|^2|\nabla A^o|^2\big)\gamma^4d\mu
  + c(c_\gamma)^4\int_{[\gamma>0]}|A^o|^2d\mu,
\end{align*}
where $c$ is an absolute constant.
\label{LMgapest2}
\end{lem}
\begin{proof}
Applying the divergence theorem to the interchange of covariant derivatives formula gives the following identity
\begin{multline}
\int_\Sigma |\nabla_{(2)}H|^2\gamma^4 d\mu + \frac14\int_\Sigma |H|^2|\nabla H|^2\gamma^4d\mu
=   \int_\Sigma |\Delta H|^2\gamma^4d\mu
\\
  + \int_\Sigma A*A^o*\nabla A^o*\nabla H\gamma^4d\mu
  + \int_\Sigma \nabla H*\nabla_{(2)}H * \nabla \gamma\ \gamma^3d\mu;
\label{EQgle2.5}
\end{multline}
which upon decomposing $A = A^o + \frac12gH$, using \eqref{EQbasicgradHgradAo} and estimating yields
\begin{multline*}
\frac12\int_\Sigma |\nabla_{(2)}H|^2\gamma^4 d\mu + \frac14\int_\Sigma |H|^2|\nabla H|^2\gamma^4d\mu
\le \int_\Sigma |\Delta H|^2\gamma^4d\mu
\\
  + \delta_1\int_\Sigma |H|^2|\nabla A^o|^2\gamma^4d\mu
  + c\int_\Sigma |A^o|^2|\nabla A^o|^2\gamma^4d\mu
  + c(c_\gamma)^2\int_\Sigma |\nabla H|^2 \gamma^2d\mu,
\end{multline*}
where $\delta_1 > 0$ and $c$ is a constant depending only on $\delta_1$.
For the convenience of the reader we provide a proof of \eqref{EQgle2.5} in Appendix A.
Combining this with Lemma \ref{LMgapest1} we have
\begin{align}
\Big(&\frac12-4\delta_2\Big)\int_\Sigma |\nabla_{(2)}H|^2\gamma^4 d\mu + \frac14\int_\Sigma |H|^2|\nabla H|^2\gamma^4d\mu
   + 2\lambda_1\int_\Sigma |\nabla H|^2 \gamma^4d\mu
\notag\\
&\le 
    \int_\Sigma \BW_{\lambda_1,\lambda_2}(f)\big(4\IP{\nabla H}{\nabla \gamma} + \gamma\Delta H\big) \gamma^3 d\mu
 +  \delta_2 \int_\Sigma |H|^4|A^o|^2\gamma^4 d\mu
\label{EQgle3}\\
&\quad
  + \delta_1\int_\Sigma |H|^2|\nabla A^o|^2\gamma^4d\mu
  + c\int_\Sigma\big(|A^o|^6 + |A^o|^2|\nabla A^o|^2\big)\gamma^4d\mu
\notag\\
&\quad
\notag + c(c_\gamma)^2\int_\Sigma |\nabla H|^2 \gamma^2d\mu,
\end{align}
where here and for the rest of the proof $c$ is a constant depending only on $\delta_1$ and $\delta_2$.
Interchanging again covariant derivatives, using the divergence theorem and estimating the result we have
\begin{align}
 &\int_\Sigma |H|^2|\nabla A^o|^2\gamma^4d\mu + \frac27\int_\Sigma |H|^4|A^o|^2\gamma^4d\mu
\le
  \int_\Sigma |H|^2|\nabla H|^2\gamma^4d\mu
\notag
\\
\label{EQgle4}
&\quad
  + c\int_\Sigma\big(|A^o|^6 + |A^o|^2|\nabla A^o|^2\big)\gamma^4d\mu
  + c(c_\gamma)^4\int_{[\gamma>0]}|A^o|^2d\mu.
\end{align}
We have also provided a proof of \eqref{EQgle4} in Appendix A.
Combining this with \eqref{EQgle3} we conclude
\begin{align*}
\Big(&\frac12-4\delta_2\Big)\int_\Sigma |\nabla_{(2)}H|^2\gamma^4 d\mu
   + \Big(\frac14-\delta_1\Big)\int_\Sigma |H|^2|\nabla H|^2\gamma^4d\mu
\\
&\qquad
   + 2\lambda_1\int_\Sigma |\nabla H|^2 \gamma^4d\mu
   + \frac27\delta_1\int_\Sigma |H|^4|A^o|^2\gamma^4d\mu
\\
&\le 
    \int_\Sigma \BW_{\lambda_1,\lambda_2}(f)\big(4\IP{\nabla H}{\nabla \gamma} + \gamma\Delta H\big) \gamma^3 d\mu
 +  \delta_2 \int_\Sigma |H|^4|A^o|^2\gamma^4 d\mu
\\
&\quad
  + c\int_\Sigma\big(|A^o|^6 + |A^o|^2|\nabla A^o|^2\big)\gamma^4d\mu
  + c(c_\gamma)^4\int_{[\gamma>0]}|A^o|^2d\mu.
\\
&\quad
  + c(c_\gamma)^2\int_\Sigma |\nabla H|^2 \gamma^2d\mu.
\end{align*}
Choosing $\delta_2 = \frac17\delta_1$ with $\delta_1 < \frac1{16}$ we absorb the second term on the right into the left
hand side and conclude the result.
\end{proof}

\begin{cor}
Suppose $f:\Sigma\rightarrow\R^3$ is a smooth immersed surface and $\gamma$ is as in \eqref{Egamma}.  Then
\begin{align*}
&\int_\Sigma \big(|\nabla_{(2)}A|^2
                + |A|^2|\nabla A|^2
                + |A|^4|A^o|^2\big)\gamma^4d\mu
   + 2\lambda_1\int_\Sigma |\nabla H|^2 \gamma^4d\mu
\\
&\le 
    c\int_\Sigma \BW_{\lambda_1,\lambda_2}(f)\big(4\IP{\nabla H}{\nabla \gamma} + \gamma\Delta H\big) \gamma^3 d\mu
  + c(c_\gamma)^2\int_\Sigma |\nabla H|^2 \gamma^2d\mu
\\
&\quad
  + c_1\int_\Sigma\big(|A^o|^6 + |A^o|^2|\nabla A^o|^2\big)\gamma^4d\mu
  + c(c_\gamma)^4\int_{[\gamma>0]}|A^o|^2d\mu,
\end{align*}
where $c$, $c_1$ are absolute constants.
\label{CYgapcor}
\end{cor}
\begin{proof}
The lowest order term is easy to handle since
\[
2|A|^4|A^o|^2 = 2\Big(|A^o|^2+\frac12|H|^2\Big)^2|A^o|^2 \le |H|^4|A^o|^2 + 4|A^o|^6.
\]
The improvement follows by integrating this estimate against $\gamma^4$ and using Lemma \ref{LMgapest2}.
The term of first order is estimated again by decomposition (noting \eqref{EQbasicgradHgradAo2}) as
\[
|A|^2|\nabla A|^2
\le 3|A^o|^2|\nabla A^o|^2
  + \frac32|H|^2|\nabla A^o|^2.
\]
The first term clearly presents no difficulty upon integration against $\gamma^4$, and for the second we may use again
estimate \eqref{EQgle4} and absorb.  Let us briefly present the details of these improvements. Combining
the above equations with Lemma \ref{LMgapest2} we have
\begin{align*}
&\int_\Sigma \big(|\nabla_{(2)}H|^2
                + |H|^2|\nabla H|^2
                \big)\gamma^4d\mu
   + 2\int_\Sigma |A|^4|A^o|^2 \gamma^4 d\mu
\\
&\qquad\qquad\ 
   + \frac14\int_\Sigma |A|^2|\nabla A|^2 \gamma^4 d\mu
   + 2\lambda_1\int_\Sigma |\nabla H|^2 \gamma^4d\mu
\\
&\le 
    c\int_\Sigma \BW_{\lambda_1,\lambda_2}(f)\big(4\IP{\nabla H}{\nabla \gamma} + \gamma\Delta H\big) \gamma^3 d\mu
  + \frac34\int_\Sigma |A^o|^2|\nabla A^o|^2 \gamma^4d\mu
\\
&\qquad\qquad\ 
  + \frac38\int_\Sigma |A|^2|\nabla A^o|^2 \gamma^4d\mu
  + 4\int_\Sigma |A^o|^6 \gamma^4d\mu
  + c(c_\gamma)^2\int_\Sigma |\nabla H|^2 \gamma^2d\mu
\\
&\qquad\qquad\ 
  + c\int_\Sigma\big(|A^o|^6 + |A^o|^2|\nabla A^o|^2\big)\gamma^4d\mu
  + c(c_\gamma)^4\int_{[\gamma>0]}|A^o|^2d\mu.
\end{align*}
Applying \eqref{EQgle4} we find
\begin{align*}
&\int_\Sigma \big(|\nabla_{(2)}H|^2
                + |H|^2|\nabla H|^2
                \big)\gamma^4d\mu
   + 2\int_\Sigma |A|^4|A^o|^2 \gamma^4 d\mu
\\
&\qquad\qquad\ 
   + \frac14\int_\Sigma |A|^2|\nabla A|^2 \gamma^4 d\mu
   + 2\lambda_1\int_\Sigma |\nabla H|^2 \gamma^4d\mu
\\
&\le 
    c\int_\Sigma \BW_{\lambda_1,\lambda_2}(f)\big(4\IP{\nabla H}{\nabla \gamma} + \gamma\Delta H\big) \gamma^3 d\mu
\\
&\qquad\qquad\ 
  + \frac38\int_\Sigma |H|^2|\nabla H|^2 \gamma^4d\mu
  + c(c_\gamma)^2\int_\Sigma |\nabla H|^2 \gamma^2d\mu
\\
&\qquad\qquad\ 
  + c\int_\Sigma\big(|A^o|^6 + |A^o|^2|\nabla A^o|^2\big)\gamma^4d\mu
  + c(c_\gamma)^4\int_{[\gamma>0]}|A^o|^2d\mu,
\end{align*}
which upon absorbing, multiplying through by 8, and estimating the left hand side from below gives
\begin{align*}
&\int_\Sigma \big(|\nabla_{(2)}H|^2
                + |A|^2|\nabla A|^2
                + |A|^4|A^o|^2\big)\gamma^4d\mu
   +  2\lambda_1\int_\Sigma |\nabla H|^2 \gamma^4d\mu
\\
&\le 
    c\int_\Sigma \BW_{\lambda_1,\lambda_2}(f)\big(4\IP{\nabla H}{\nabla \gamma} + \gamma\Delta H\big) \gamma^3 d\mu
  + c(c_\gamma)^2\int_\Sigma |\nabla H|^2 \gamma^2d\mu
\\
&\quad
  + c\int_\Sigma\big(|A^o|^6 + |A^o|^2|\nabla A^o|^2\big)\gamma^4d\mu
  + c(c_\gamma)^4\int_{[\gamma>0]}|A^o|^2d\mu,
\end{align*}
as desired.

Interchanging covariant derivatives, using the divergence theorem and then estimating gives
\begin{align}
\int_\Sigma |\nabla_{(2)}A^o|^2\gamma^4d\mu
 &\le
    \int_\Sigma |\Delta A^o|^2\gamma^4d\mu 
  + c\int_\Sigma |A|^2|\nabla A|^2\gamma^4d\mu
\notag
\\&\quad
  + \int_\Sigma \nabla_{(2)}A^o*\nabla A^o*\nabla\gamma\ \gamma^3d\mu
\notag
\\
 &\le
    \frac12\int_\Sigma |\nabla_{(2)}A^o|^2\gamma^4d\mu
  + c\int_\Sigma |\nabla_{(2)}H|^2\gamma^4d\mu
\notag
\\
&\quad
  + c\int_\Sigma \big(|A|^2|\nabla A|^2
               + |A|^4|A^o|^2\big)\gamma^4d\mu
\notag
\\
&\quad
  + c(c_\gamma)^2\int_\Sigma |\nabla A^o|^2\gamma^2d\mu.
\label{EQgle8}
\end{align}
The interested reader can find a proof of \eqref{EQgle8} in Appendix A.
To deal with the last term we interchange again covariant derivatives, apply \eqref{EQbasicgradHgradAo} and estimate to obtain
\begin{align}
(c_\gamma)^2&\int_\Sigma |\nabla A^o|^2\gamma^2d\mu
\notag\\
&\le (c_\gamma)^2\int_\Sigma |\nabla H|^2\gamma^2d\mu
  + 2\int_\Sigma |A^o|^6\gamma^4d\mu
  + c(c_\gamma)^4\int_{[\gamma>0]}|A^o|^2d\mu.
\label{EQgle9}
\end{align}
We have provided a proof of \eqref{EQgle9} in Appendix A.
Combining these estimates with Lemma \ref{LMgapest2} and applying \eqref{EQbasicgradHgradAo2} finishes the proof.
\end{proof}

\begin{prop}
Suppose $f:\Sigma\rightarrow\R^3$ is a smooth immersed surface satisfying \eqref{Egapass} and $\gamma$ is as in \eqref{Egamma}.
Then
\begin{align}
\int_\Sigma &\big(|\nabla_{(2)}A|^2
                + |A|^2|\nabla A|^2
                + |A|^4|A^o|^2\big)\gamma^4d\mu
 + 2\lambda_1\int_\Sigma |\nabla H|^2\gamma^4d\mu
\notag\\*
 &\!\!\!\!\le c\int_\Sigma \BW^2_{\lambda_1,\lambda_2}(f) \gamma^4 d\mu
 + (c_\gamma)^2\int_\Sigma |\nabla H|^2\gamma^2d\mu
 + c(c_{\gamma})^4\vn{A^o}^2_{2,[\gamma>0]},
 \label{Egapest1}
\end{align}
where $c$ is an absolute constant.
\label{Pgapest}
\end{prop}
\begin{proof}
We shall estimate each of the terms on the right hand side of Corollary \ref{CYgapcor} in turn.  First,
\begin{multline}
    \int_\Sigma \BW_{\lambda_1,\lambda_2}(f)\big(\IP{\nabla H}{\nabla \gamma} + \gamma\Delta H\big) \gamma^3 d\mu
\\
\le \frac1{2\delta}\int_\Sigma \BW^2_{\lambda_1,\lambda_2}(f) \gamma^4 d\mu + \delta(c_\gamma)^2\int_\Sigma |\nabla
H|^2\gamma^2d\mu + \delta \int_\Sigma |\Delta H|^2\gamma^4d\mu,
\label{EQgle5}
\end{multline}
where $\delta > 0$.
Next we need a consequence of the Michael-Simon Sobolev inequality first proven in \cite{KS01}: For any smooth closed
immersed surface $\Sigma$ in $\R^3$ there are absolute constants $c$, $c_2$ such that
\begin{align}
\int_\Sigma &\big(|\nabla A^o|^2|A^o|^2 + |A^o|^6)\gamma^4 d\mu
\notag\\
&\le
  c_2\vn{A^o}^2_{2,[\gamma>0]}\int_\Sigma \big( |\nabla_{(2)}A|^2 + |\nabla A|^2|A|^2 + |A|^4|A^o|^2 \big)\gamma^4 d\mu
\notag\\
&\hskip+3cm
+ c(c_\gamma)^4\vn{A^o}^4_{2,[\gamma>0]}.
\label{EQmssforAo}
\end{align}
This, combined with \eqref{Egapass}, \eqref{EQgle5} and Corollary \ref{CYgapcor} implies
\begin{align*}
(1-c_1c_2\varepsilon_1&-c\delta)
\int_\Sigma \big(|\nabla_{(2)}A|^2
                + |A|^2|\nabla A|^2
                + |A|^4|A^o|^2\big)\gamma^4d\mu
\\
 &+ 2\lambda_1\int_\Sigma |\nabla H|^2\gamma^4d\mu
 - c\delta(c_\gamma)^2\int_\Sigma |\nabla H|^2\gamma^2d\mu
\\
 &\hskip-1cm\le c\int_\Sigma \BW^2_{\lambda_1,\lambda_2}(f) \gamma^4 d\mu + c(c_{\gamma})^4\vn{A^o}^2_{2,[\gamma>0]},
\end{align*}
which for $\delta < \frac1{c}(1-c_1c_2\varepsilon_1)$ gives the result.
\end{proof}

\section{Proof of the main result}

Suppose $\BW_{\lambda_1,\lambda_2}(f) \equiv 0$.
Estimate \eqref{Egapest1} implies (recall that $\lambda_1 \ge 0$)
\begin{align}
\int_\Sigma \big(|\nabla_{(2)}A|^2
                + |A|^2|\nabla A|^2
                &+ |A|^4|A^o|^2\big)\gamma^4d\mu
\notag\\
  &\le
  c(c_\gamma)^2\int_\Sigma |\nabla H|^2\gamma^2d\mu
+ c(c_\gamma)^4\vn{A^o}^2_{2,{[\gamma>0]}}.
\label{EQgle6}
\end{align}
Using the divergence theorem we estimate
\begin{align*}
   \int_\Sigma |\nabla A^o|^2\gamma^2d\mu
&\le
   \int_\Sigma |\Delta A^o|\,|A^o|\,\gamma^2d\mu
+  c(c_\gamma)\int_\Sigma |\nabla A^o|\,|A^o|\,\gamma d\mu
\\
&\le
   \frac1{2}\int_\Sigma |\nabla A^o|^2\gamma^2d\mu
 + \frac1{2(c_\gamma)^2}\delta\int_\Sigma |\Delta A^o|^2\gamma^4d\mu
\\
&\qquad
 + c(c_\gamma)^2\vn{A^o}^2_{2,[\gamma>0]}
\end{align*}
where $\delta>0$ and $c$ is a constant depending only on $\delta$.
Combining this with \eqref{EQbasicgradHgradAo}, we have 
\begin{align*}
   (c_\gamma)^2\int_\Sigma |\nabla H|^2\gamma^2d\mu
&\le
  4(c_\gamma)^2\int_\Sigma |\nabla A^o|^2\gamma^2d\mu
\\
&\le
   2\delta\int_\Sigma |\Delta A^o|^2\gamma^4d\mu
 + c(c_\gamma)^4\vn{A^o}^2_{2,[\gamma>0]}.
\end{align*}
We now set the cutoff function $\gamma$ to be such that $\gamma(p) = \varphi\left(\frac{1}{\rho}|f(p)|\right)$, where
$\varphi\in C^1(\R)$ and $\chi_{B_{1/2}(0)} \le \varphi \le \chi_{B_1(0)}$.
With this choice ${c_{\gamma}} = \frac{c}{\rho}$.
In view of \eqref{EQgle6}, absorbing $\int_\Sigma |\Delta A^o|^2\gamma^4d\mu$ on the left gives
\begin{equation*}
\int_{f^{-1}(B_{\rho/2}(0))} \big(|\nabla_{(2)}A|^2
                + |A|^2|\nabla A|^2
                + |A|^4|A^o|^2\big)d\mu
  \le c\frac{1}{\rho^4}\vn{A^o}^2_{2,{f^{-1}(B_\rho(0))}};
\end{equation*}
which upon taking $\rho\rightarrow\infty$ yields
\[
|\nabla_{(2)}A|^2
                + |A|^2|\nabla A|^2
                + |A|^4|A^o|^2
\equiv 0.
\]
This gives us a lot of information about the kind of critical surface we have.  In particular,
\begin{equation}
|A| \equiv 0\text{, }|A^o| \equiv 0\text{, or both.}
\label{EQgle7}
\end{equation}
This implies that $f(\Sigma)$ is an embedded flat plane or round sphere.  The particular cases for various values of the
weights $\lambda_1,\lambda_2$ follow from observing the structure of the Euler-Lagrange equation
\begin{equation}
\BW_{\lambda_1,\lambda_2}(f)
 = \Delta H + H|A^o|^2 - 2\lambda_1H - 2\lambda_2 = 0.
\label{Eeulerlagrangeproofgap}
\end{equation}
For $\lambda_2 < 0$, \eqref{Eeulerlagrangeproofgap} is not satisfied by flat planes, and among the round spheres
it is only satisfied by $S_{\frac{2\lambda_1}{\lambda_2}}(0)$ and translations thereof.  For $\lambda_2 = 0$,
\eqref{Eeulerlagrangeproofgap} is not satisfied for any round sphere.
In this case we only have flat planes.
If $\lambda_2 > 0$, \eqref{Eeulerlagrangeproofgap} is not satisfied for either flat planes or round spheres, in
contradiction with the assumption $\BW_{\lambda_1,\lambda_2}(f) \equiv 0$.  In this case, there does not exist a
critical point of $\SW_{\lambda_1,\lambda_2}$ satisfying \eqref{Egapass}.
For $\lambda_1 = 0$, \eqref{Eeulerlagrangeproofgap} is only satisfied by flat planes and round spheres (of any
radius) when $\lambda_2$ is also equal to zero.

It remains to show that the image of $f$ consists of only one fixed round sphere or plane.
Denote the image of $f$ by $S\subset\R^3$.
The conclusion follows since $f$ is complete, and so the map $f:(\Sigma,g)\rightarrow S$ is a global isometry.
This finishes the proof.


\appendix
\section{Selected proofs}

We collect here the proofs of several well-known formulae and results for the convenience of the reader and readability
of the paper.
Many of the statements contained in this appendix have appeared in a similar form in \cite{KS01,KS02}.

\begin{proof}[ of \eqref{EQgle2.5}, cf. {\cite[{\it Lem.} 2.3]{KS01}}.]
Let us first prove that
\begin{equation}
\label{EQapp1}
\nabla \Delta H = \Delta \nabla H - \frac14H^2\nabla H + A*A^o*\nabla H.
\end{equation}
Equation \eqref{EQinterchangespecific} implies
\[
\nabla_{ijk}H - \nabla_{jik}H
 = (A_{lj}A_{ik} - A_{il}A_{jk})g^{lm}\nabla_mH,
\]
so
\begin{equation}
\nabla_{j}\Delta H
 = \Delta\nabla_{j}H
   - (A^m_{j}H - A^m_{p}A^p_{j})\nabla_mH.
\label{EQapp2}
\end{equation}
Note that
\begin{align*}
A^m_j\nabla_mH
 &= \Big( (A^o)^m_j + \frac12g^m_jH\Big)\nabla_mH
  = \frac12H\nabla_jH + A^o*\nabla H
\\
A^m_pA^p_j\nabla_mH
 &= \Big( (A^o)^p_j + \frac12g^p_jH\Big)\Big( (A^o)^m_p + \frac12g^m_pH\Big)\nabla_mH
\\
 &= \frac14H\nabla_jH + A^o*A^o*\nabla H + H A^o*\nabla H,
\end{align*}
which when combined with \eqref{EQapp2} above gives \eqref{EQapp1}.
Testing \eqref{EQapp1} agains $\nabla H\,\gamma^4$ and integrating yields
\begin{align}
\int_\Sigma \IP{\nabla\Delta H}{\nabla H}_g \gamma^4 d\mu
 &=
    \int_\Sigma \IP{\Delta\nabla H}{\nabla H}_g \gamma^4 d\mu
\notag
\\
\label{EQapp3}
&\qquad
  - \frac14 \int_\Sigma |H|^2|\nabla H|^2\gamma^4d\mu
  + \int_\Sigma A*A^o*\nabla H*\nabla H\, \gamma^4d\mu.
\end{align}
Using the divergence theorem we have
\[
\int_\Sigma \IP{\nabla\Delta H}{\nabla H}_g \gamma^4 d\mu
 = -\int_\Sigma |\Delta H|^2\gamma^4d\mu - 4\int_\Sigma \Delta H \IP{\nabla H}{\nabla\gamma}_g\gamma^3d\mu
\]
and
\[
\int_\Sigma \IP{\Delta\nabla H}{\nabla H}_g \gamma^4 d\mu
 = -\int_\Sigma |\nabla_{(2)}H|^2\gamma^4d\mu - 4\int_\Sigma \IP{\nabla_{(2)} H}{\nabla\gamma\nabla H}_g\gamma^3d\mu.
\]
Combining these identities with \eqref{EQapp3} we obtain
\begin{align*}
  \int_\Sigma |\nabla_{(2)}H|^2\gamma^4d\mu
&+ \frac14 \int_\Sigma |H|^2|\nabla H|^2\gamma^4d\mu
=
    \int_\Sigma |\Delta H|^2\gamma^4d\mu
\\
&\qquad
  + \int_\Sigma \nabla_{(2)} H*\nabla H*\nabla \gamma\ \gamma^3d\mu
  + \int_\Sigma A*A^o*\nabla H*\nabla H\ \gamma^4d\mu,
\end{align*}
which upon noting \eqref{EQbasicgradHgradAo}, proves \eqref{EQgle2.5}.
\end{proof}

\begin{proof}[ of \eqref{EQgle4}, cf. {\cite[{\it eqn. }(17)]{KS01}}.]
We shall prove
\begin{align}
 \Big(1-\delta\Big)\int_\Sigma |H|^2&|\nabla A^o|^2\gamma^4d\mu
 + \Big(\frac12-2\delta\Big)\int_\Sigma |H|^4|A^o|^2\gamma^4d\mu
\notag
\\
&\le
 \Big(\frac12+3\delta\Big)\int_\Sigma |H|^2|\nabla H|^2\gamma^4d\mu
\notag
\\
\label{EQapp4}
&\quad
  + c\int_\Sigma\big(|A^o|^6 + |A^o|^2|\nabla A^o|^2\big)\gamma^4d\mu
  + c(c_\gamma)^4\int_{[\gamma>0]}|A^o|^2d\mu,
\end{align}
for $\delta>0$, 
where $c$ is a constant depending only on $\delta$.
Equation \eqref{EQgle4} follows from \eqref{EQapp4} with the choice $\delta=\frac18$.
We first begin with the following consequence of Simons' identity:
\begin{equation}
\label{EQapp5}
\Delta A^o = S^o(\nabla_{(2)}H) + \frac12H^2A^o - |A^o|^2A^o.
\end{equation}
This follows readily from \eqref{EQsi},
\begin{align*}
\Delta A^o_{ij} &= \Delta A_{ij} - \frac12g_{ij}H
\\
                &= \nabla_{ij}H - \frac12g_{ij}\Delta H + HA_i^jA_{lj} - |A|^2A_{ij}
\\
                &= S^o(\nabla_{(2)}H) + HA_i^jA_{lj} - |A|^2A_{ij}
\\
                &= S^o(\nabla_{(2)}H) + 2KA^o_{ij},
\end{align*}
provided we show
\begin{equation}
\label{EQapp6}
HA_i^jA_{lj} - |A|^2A_{ij} = 2KA^o_{ij}.
\end{equation}
Choosing normal coordinates so that $A$ is diagonalised (at a point) with $A=\delta_i^jk_j$, $H=k_1+k_2$,
$K=k_1k_2$, $|A|^2=k_2^2+k_1^2$ (at this point) and
\[
A^o_{ij} = \begin{cases} \frac12(k_1-k_2),\quad\text{ if $i=j=1$}\\
                         \frac12(k_2-k_1),\quad\text{ if $i=j=2$}\\
                      0,\quad\text{ otherwise},
\end{cases}
\]
we have
\begin{align*}
(i=j=1)\qquad\qquad HA_1^jA_{l1} - |A|^2A_{i1}
 &= (k_1+k_2)k_1^2 - (k_1^2+k_2^2)k_1
\\
 &= \frac12k_1k_2(k_1-k_2)
 = 2KA^o_{11}
\\
(i=j=2)\qquad\qquad HA_2^jA_{l2} - |A|^2A_{i2}
 &= (k_1+k_2)k_2^2 - (k_1^2+k_2^2)k_2
\\
 &= \frac12k_1k_2(k_2-k_1)
 = 2KA^o_{22},
\end{align*}
and otherwise \eqref{EQapp6} holds trivially.  Therefore \eqref{EQapp6} is proved.
This also proves \eqref{EQapp5}, since
\[
K = k_1k_2 =  \frac14(k_1+k_2)^2-\frac14(k_1-k_2)^2 = \frac14H^2-\frac12|A^o|^2.
\]
Now employing the divergence theorem and inserting \eqref{EQapp5} we compute
\begin{align*}
\int_\Sigma &|H|^2|\nabla A^o|^2\gamma^4d\mu
\\
&= - \int_\Sigma |H|^2\IP{A^o}{\Delta A^o}_g\gamma^4d\mu
\\
&\qquad
   - 2\int_\Sigma H\IP{\nabla A^o}{\nabla H\, A^o}_g\gamma^4d\mu
   - 4\int_\Sigma |H|^2\IP{\nabla A^o}{\nabla\gamma\, A^o}_g\gamma^3d\mu
\\
&= - \int_\Sigma |H|^2\IP{A^o}{S^o(\nabla_{(2)}H) + \frac12|H|^2A^o - |A^o|^2A^o}_g\gamma^4d\mu
\\
&\qquad
   - 2\int_\Sigma H\IP{\nabla A^o}{\nabla H\, A^o}_g\gamma^4d\mu
   - 4\int_\Sigma |H|^2\IP{\nabla A^o}{\nabla\gamma\, A^o}_g\gamma^3d\mu
\\
&= - \int_\Sigma |H|^2\IP{A^o}{\nabla_{(2)}H}_g\gamma^4d\mu
   - \frac12\int_\Sigma |H|^4|A^o|^2\gamma^4d\mu
   + \int_\Sigma |H|^2|A^o|^4\gamma^4d\mu
\\
&\qquad
   - 2\int_\Sigma H\IP{\nabla A^o}{\nabla H\, A^o}_g\gamma^4d\mu
   - 4\int_\Sigma |H|^2\IP{\nabla A^o}{\nabla\gamma\, A^o}_g\gamma^3d\mu
\\
&=  \int_\Sigma |H|^2\IP{\nabla^* A^o}{\nabla H}_g\gamma^4d\mu
  + 2\int_\Sigma H\IP{A^o}{\nabla H\nabla H}_g\gamma^4d\mu
\\
&\qquad
   - \frac12\int_\Sigma |H|^4|A^o|^2\gamma^4d\mu
   + \int_\Sigma |H|^2|A^o|^4\gamma^4d\mu
   - 2\int_\Sigma H\IP{\nabla A^o}{\nabla H\, A^o}_g\gamma^4d\mu
\\
&\qquad
   + 4\int_\Sigma |H|^2\IP{A^o}{\nabla\gamma\nabla H}_g\gamma^3d\mu
   - 4\int_\Sigma |H|^2\IP{\nabla A^o}{\nabla\gamma\, A^o}_g\gamma^3d\mu.
\end{align*}
Noting \eqref{EQbasicgradHgradAo}, we estimate the right hand side to obtain for $\delta > 0$
\begin{align*}
\int_\Sigma &|H|^2|\nabla A^o|^2\gamma^4d\mu
   + \frac12\int_\Sigma |H|^4|A^o|^2\gamma^4d\mu
\\
&=  \frac12\int_\Sigma |H|^2|\nabla H|^2\gamma^4d\mu
  + 2\int_\Sigma H\IP{A^o}{\nabla H\nabla H}_g\gamma^4d\mu
\\
&\qquad
   + \int_\Sigma |H|^2|A^o|^4\gamma^4d\mu
   - 2\int_\Sigma H\IP{\nabla A^o}{\nabla H\, A^o}_g\gamma^4d\mu
\\
&\qquad
   + 4\int_\Sigma |H|^2\IP{A^o}{\nabla\gamma\nabla H}_g\gamma^3d\mu
   - 4\int_\Sigma |H|^2\IP{\nabla A^o}{\nabla\gamma\, A^o}_g\gamma^3d\mu
\\
&\le \Big(\frac12+3\delta\Big)\int_\Sigma |H|^2|\nabla H|^2\gamma^4d\mu
  + \frac1{\delta}\int_\Sigma |A^o|^2|\nabla H|^2\gamma^4d\mu
\\
&\qquad
  + \delta\int_\Sigma |H|^4|A^o|^2\gamma^4d\mu
  + \frac1{4\delta}\int_\Sigma \big(|A^o|^6 + 4|A^o|^2|\nabla A^o|^2\big)\gamma^4d\mu
\\
&\qquad
  + \delta\int_\Sigma |H|^2|\nabla A^o|^2\gamma^4d\mu
  + (c_\gamma)^2\frac8{\delta}\int_\Sigma |H|^2|A^o|^2\gamma^2d\mu
\\
&\le \Big(\frac12+3\delta\Big)\int_\Sigma |H|^2|\nabla H|^2\gamma^4d\mu
  + \delta\int_\Sigma |H|^2|\nabla A^o|^2\gamma^4d\mu
\\
&\qquad
  + 2\delta\int_\Sigma |H|^4|A^o|^2\gamma^4d\mu
  + \frac1{4\delta}\int_\Sigma \big(|A^o|^6 + 20|A^o|^2|\nabla A^o|^2\big)\gamma^4d\mu
\\
&\qquad
  + (c_\gamma)^4\frac{16}{\delta^3}\int_{[\gamma>0]}|A^o|^2d\mu.
\end{align*}
Absorbing the second and third terms from the right hand side into the left finishes the proof of \eqref{EQapp4}.
\end{proof}

\begin{proof}[ of \eqref{EQgle8}, cf. {\cite[{\it Prop. }2.4]{KS01}}.]
Let us first note that \eqref{EQbasicgradHgradAo} implies $2\nabla^*(\nabla^*A^o) = \Delta H \le 2|\nabla_{(2)}H|^2$, so we may estimate
\begin{align}
    \int_\Sigma |\Delta A^o|^2\gamma^4d\mu 
  &+ c\int_\Sigma |A|^2|\nabla A|^2\gamma^4d\mu
  + \int_\Sigma \nabla_{(2)}A^o*\nabla A^o*\nabla\gamma\ \gamma^3d\mu
\notag
\\
 &\le
    c\int_\Sigma |\nabla_{(2)}H|^2\gamma^4d\mu
  + \frac12\int_\Sigma |\nabla_{(2)}A^o|^2\gamma^4d\mu
\notag
\\
&\quad
  + c\int_\Sigma \big(|A|^2|\nabla A|^2
               + |A|^4|A^o|^2\big)\gamma^4d\mu
  + c(c_\gamma)^2\int_\Sigma |\nabla A^o|^2\gamma^2d\mu.
\notag
\end{align}
Thus \eqref{EQgle8} will be proved provided we show
\begin{align}
\int_\Sigma |\nabla_{(2)}A^o|^2\gamma^4d\mu
 &\le
    \int_\Sigma |\Delta A^o|^2\gamma^4d\mu 
  + c\int_\Sigma |A|^2|\nabla A|^2\gamma^4d\mu
\notag
\\
&\hskip+2cm
  + \int_\Sigma \nabla_{(2)}A^o*\nabla A^o*\nabla\gamma\ \gamma^3d\mu.
\label{EQapp7}
\end{align}
This again relies upon interchange of covariant derivatives; we shall prove the following identity
\begin{equation}
\label{EQapp8}
\Delta \nabla A^o = \nabla \Delta A^o + \nabla A^o * A * A .
\end{equation}
Using \eqref{EQinterchangegeneral}, we compute
\begin{align}
\nabla_{ijk}A^o_{lm}
 &= \nabla_{ikj}A^o_{lm} + \nabla (A^o*A*A)
\notag\\
 &= \nabla_{kij}A^o_{lm} +  \nabla A^o*A*A + A^o * \nabla A*A*A
\notag\\
\label{EQapp9}
 &= \nabla_{kij}A^o_{lm} +  \nabla A^o*A*A,
\end{align}
where in the last equality we used
\[
 \nabla_i A_{jk}
 = \nabla_i A^o_{jk} + \frac12g_{jk}\nabla_i H
 = \nabla_i A^o_{jk} + g_{jk}\nabla_p (A^o)_i^p.
\]
Tracing \eqref{EQapp9} with $g^{ij}$ gives \eqref{EQapp8}, which testing against $\nabla A^o\,\gamma^4$ and using the
divergence theorem implies
\begin{align}
\int_\Sigma |\nabla_{(2)}A^o|^2\gamma^4d\mu
 &=
 - \int_\Sigma \IP{\nabla A^o}{\Delta \nabla A^o}_g\gamma^4d\mu
 - 4\int_\Sigma \IP{\nabla\gamma\nabla A^o}{\nabla_{(2)}A^o}_g \gamma^3d\mu
\notag
\\
 &=
 - \int_\Sigma \IP{\nabla A^o}{\nabla\Delta  A^o}_g\gamma^4d\mu
 - 4\int_\Sigma \IP{\nabla\gamma\nabla A^o}{\nabla_{(2)}A^o}_g \gamma^3d\mu
\notag
\\
 &\quad + \int_\Sigma \nabla A^o * \nabla A^o * A * A\ \gamma^4d\mu
\notag
\\
 &=
  \int_\Sigma |\Delta  A^o|^2\gamma^4d\mu
 + 4\int_\Sigma \IP{\nabla A^o}{\nabla\gamma\Delta A^o}_g \gamma^3d\mu
\notag
\\
 &\quad
 - 4\int_\Sigma \IP{\nabla\gamma\nabla A^o}{\nabla_{(2)}A^o}_g \gamma^3d\mu
 + \int_\Sigma \nabla A^o * \nabla A^o * A * A\ \gamma^4d\mu
\notag
\\
 &\le
    \int_\Sigma |\Delta A^o|^2\gamma^4d\mu 
  + c\int_\Sigma |A|^2|\nabla A^o|^2\gamma^4d\mu
\notag
\\
 &\quad
  + \int_\Sigma \nabla_{(2)}A^o*\nabla A^o*\nabla\gamma\ \gamma^3d\mu
  + \frac12\int_\Sigma |\nabla_{(2)}A^o|^2\gamma^4d\mu,
\notag
\end{align}
which clearly implies \eqref{EQapp7}.
\end{proof}

\begin{proof}[ of \eqref{EQgle9}.]
Integrating \eqref{EQapp5} against $A^o\gamma^2$ and using the divergence theorem we have
\begin{align*}
\int_\Sigma |\nabla A^o|^2\gamma^2d\mu
 &= -\int_\Sigma \IP{A^o}{\Delta A^o}_g\gamma^2d\mu
    -2\int_\Sigma \nabla A^o*A^o*\nabla\gamma\, \gamma d\mu
\\
 &= -\int_\Sigma \IP{A^o}{\nabla_{(2)}H + \frac12H^2A^o - |A^o|^2A^o}_g\gamma^2d\mu
\\
&\quad
    -2\int_\Sigma \nabla A^o*A^o*\nabla\gamma\, \gamma d\mu
\\
 &= \int_\Sigma \IP{\nabla^*A^o}{\nabla H}_g\gamma^2d\mu
    - \frac12\int_\Sigma H^2|A^o|^2\gamma^2d\mu
    + \int_\Sigma |A^o|^4\gamma^2d\mu
\\
&\quad
    + 2\int_\Sigma \nabla H*A^o*\nabla\gamma\, \gamma d\mu
    -2\int_\Sigma \nabla A^o*A^o*\nabla\gamma\, \gamma d\mu.
\end{align*}
Applying \eqref{EQbasicgradHgradAo} and estimating we obtain
\begin{align}
\int_\Sigma &|\nabla A^o|^2\gamma^2d\mu
    + \frac12\int_\Sigma H^2|A^o|^2\gamma^2d\mu
\label{EQintgradHgradAoapp}
\\
 &= \frac12\int_\Sigma |\nabla H|^2\gamma^2d\mu
    + \int_\Sigma |A^o|^4\gamma^2d\mu
    + 4\int_\Sigma \nabla^*A^o*A^o*\nabla\gamma\, \gamma d\mu
\notag\\
&\quad
    -2\int_\Sigma \nabla A^o*A^o*\nabla\gamma\, \gamma d\mu
\notag
\\
 &\le \frac12\int_\Sigma |\nabla H|^2\gamma^2d\mu
    + \frac12\int_\Sigma |\nabla A^o|^2\gamma^2d\mu
    + \int_\Sigma |A^o|^4\gamma^2d\mu
\notag\\
&\quad
    + c(c_\gamma)^2\int_{[\gamma>0]} |A^o|^2 d\mu
\notag
\\
 &\le \frac12\int_\Sigma |\nabla H|^2\gamma^2d\mu
    + \frac12\int_\Sigma |\nabla A^o|^2\gamma^2d\mu
    + \frac1{(c_\gamma)^2}\int_\Sigma |A^o|^6\gamma^4d\mu
\notag\\
&\quad
    + c(c_\gamma)^2\int_{[\gamma>0]} |A^o|^2 d\mu,
\notag
\end{align}
which, upon absorption and multiplication by $(c_\gamma)^2$, implies \eqref{EQgle9}.
\end{proof}


\bibliographystyle{abbrv}
\bibliography{helfrichsurfaces}

\begin{thebibliography}{10}

\bibitem{BdCE88}
J.~Barbosa, M.~Do~Carmo, and J.~Eschenburg.
\newblock Stability of hypersurfaces of constant mean curvature in riemannian
  manifolds.
\newblock {\em Math. Z.}, 197(1):123--138, 1988.

\bibitem{B80sophie}
L.~Bucciarelli and N.~Dworsky.
\newblock {\em Sophie Germain: An essay in the history of the theory of
  elasticity}.
\newblock Springer, 1980.

\bibitem{DFGS11}
A.~Dall'Acqua, S.~Fr{\"o}hlich, H.-C. Grunau, and F.~Schieweck.
\newblock Symmetric {W}illmore surfaces of revolution satisfying arbitrary
  {D}irichlet boundary data.
\newblock {\em Adv. Calc. Var.}, 4(1):1--81, 2011.

\bibitem{DG09}
K.~Deckelnick and H.-C. Grunau.
\newblock A {N}avier boundary value problem for {W}illmore surfaces of
  revolution.
\newblock {\em Analysis}, 29(3):229--258, 2009.

\bibitem{E04}
K.~Ecker.
\newblock {\em Regularity theory for mean curvature flow}, volume~57 of {\em
  Progress in nonlinear differential equations and their applications}.
\newblock Birkhauser, 2004.

\bibitem{G21}
S.~Germain.
\newblock {\em Recherches sur la th\`eorie des surfaces \'elastiques}.
\newblock Imprimerie de Huzard-Courcier, Paris, 1921.

\bibitem{H82}
R.~Hamilton.
\newblock {Three-manifolds with positive Ricci curvature}.
\newblock {\em J. Differential Geom.}, 17:255--306, 1982.

\bibitem{H73}
W.~Helfrich.
\newblock Elastic properties of lipid bilayers: theory and possible
  experiments.
\newblock {\em Z. Naturforsch.}, 28(11):693--703, 1973.

\bibitem{KS01}
E.~Kuwert and R.~Sch{\"a}tzle.
\newblock {The Willmore flow with small initial energy}.
\newblock {\em J. Differential Geom.}, 57(3):409--441, 2001.

\bibitem{KS02}
E.~Kuwert and R.~Sch{\"a}tzle.
\newblock {Gradient flow for the Willmore functional}.
\newblock {\em Comm. Anal. Geom.}, 10(2):307--339, 2002.

\bibitem{KS04}
E.~Kuwert and R.~Sch{\"a}tzle.
\newblock {Removability of point singularities of Willmore surfaces}.
\newblock {\em Ann. of Math.}, 160(1):315--357, 2004.

\bibitem{MS73}
J.~Michael and L.~Simon.
\newblock Sobolev and mean-value inequalities on generalized submanifolds of
  {$\R^n$}.
\newblock {\em Communications on Pure and Applied Mathematics}, 26(3):361--379,
  1973.

\bibitem{P1812}
S.~Poisson.
\newblock M\'emoire sur les surfaces \'elastiques.
\newblock {\em Cl. Sci. Math\'em. Phys. Inst. de France (2nd printing)}, pages
  167--225, 1812.

\bibitem{S22}
W.~Schadow, 1922.
\newblock See footnote on page 56 in \cite{T23}.

\bibitem{T23}
G.~Thomsen.
\newblock {\"U}ber konforme {G}eometrie {I}: {G}rundlagen der konformen
  {F}l{\"a}chentheorie. ({G}erman).
\newblock {\em Abh. Math. Sem. Hamburg}, 3:31--56, 1923.

\bibitem{W10sd}
G.~Wheeler.
\newblock Surface diffusion flow near spheres.
\newblock {\em Calc. Var. Partial Differential Equations}, pages 1--21, 2011.
\newblock 10.1007/s00526-011-0429-4.

\bibitem{W65}
T.~Willmore.
\newblock Note on embedded surfaces.
\newblock {\em An. St. Univ. Iasi, Mat. 12B}, 493:496, 1965.

\bibitem{W93}
T.~Willmore.
\newblock {\em Riemannian geometry}.
\newblock Clarendon {P}ress Oxford, England, 1993.

\end{thebibliography}
\end{document}